\documentclass[a4paper,reqno,12pt]{amsart}

\usepackage{amsmath, amssymb, amsthm, enumerate,amsfonts,latexsym, mathrsfs}

\usepackage[top=30mm,right=27mm,bottom=30mm,left=27mm]{geometry}

\newtheorem{theorem}{Theorem}[section]
\newtheorem{corollary}[theorem]{Corollary}
\newtheorem{lemma}[theorem]{Lemma}

\makeatletter
\newcommand{\imod}[1]{\allowbreak\mkern4mu({\operator@font mod}\,\,#1)}
\makeatother

\newtheorem*{thmA}{Theorem~A}
\newtheorem*{thmB}{Theorem~B}
\newtheorem*{corC}{Corollary~C}

\newcommand{\Irr}{{\mathrm {Irr}}}

\newcommand{\Aut}{{\mathrm {Aut}}}

\newcommand{\PSL}{{\mathrm {PSL}}}

\newcommand{\SSS}{\mathrm{S}}
\newcommand{\A}{\mathrm{A}}

\newcommand{\Atlas}{{\sf Atlas}}
\newcommand{\GAP}{{\sf GAP}}
\newcommand{\Van}{{\mathrm {Van}}}
\newcommand{\ord}{{\mathrm {ord}}}
\newcommand{\vC}{v\mathcal{C}}

\theoremstyle{definition}

\begin{document}
\title[\textbf{Restriction on the vanishing set}]{\textbf{Finite groups with some restriction on the vanishing set}}

\author{Sesuai Y. Madanha}
\address{School of Mathematics Statistics and Computer Science, University of KwaZulu-Natal, Durban, South Africa}
\address{Permanent Address: Department of Mathematics and Applied Mathematics, University of Pretoria, Private Bag X20, Hatfield, Pretoria 0028, South Africa}
\email{sesuai.madanha@up.ac.za}

\author{Bernardo G. Rodrigues}
\address{School of Mathematics Statistics and Computer Science, University of KwaZulu-Natal, Durban, South Africa}
\address{Permanent Address: Department of Mathematics and Applied Mathematics, University of Pretoria, Private Bag X20, Hatfield, Pretoria 0028, South Africa}
\email{bernardo.rodrigues@up.ac.za}

\thanks{}


\date{\today}


\begin{abstract}
Let $ x $ be an element of a finite group $ G $ and denote the order of $ x $ by $ \ord(x) $. We consider a finite group $ G $ such that $ \gcd(\ord(x),\ord(y))\leqslant 2 $ for any two vanishing elements $ x $ and $ y $ contained in distinct conjugacy classes. We show that such a group $ G $ is solvable. When $ G $ with the property above is supersolvable, we show that $ G $ has a normal metabelian $ 2 $-complement.
\end{abstract}

\maketitle


\section{Introduction}

Let $ G $ be a finite group. An element $ g\in G $ is a vanishing element if there exists an irreducible character $ \chi $ of $ G $ such that $ \chi(g)=0 $. The set of all vanishing elements of $ G $ is denoted by $ \Van(G)$. A classical theorem of Burnside \cite[Theorem 3.15]{Isa06} implies that $ \Van(G) $ is non-empty when $ G $ is non-abelian. Note that
\begin{center}
$ \Van(G)=\bigcup_{i=1}^{r} \vC_{i}  $,
\end{center}
where each $ \vC_{i} $ is a vanishing conjugacy class. We denote the order of the elements contained in a vanishing conjugacy class $ \vC_{i} $ by $ \ord(v\mathcal{C}_{i}) $. Many authors have studied finite groups $ G $ with certain restrictions on the set $ \Van(G)$ (see \cite{DPSS09,BDS09, DPSS10a,DPSS10b,ZLS11,Rob19}). We shall discuss some of that work here. Let $ p $ be a fixed prime. Dolfi, Pacifici, Sanus and Spiga in \cite{DPSS09} studied finite groups $ G $ such that $ \ord (\vC_{i})\neq p^{n} $ for some $ n $, for all $ i\in \{1, 2, \dots, r \} $. They showed that $ G $ has a normal Sylow $ p $-subgroup \cite[Theorem A]{DPSS09}. On the other hand, in \cite{BDS09}, the authors studied finite groups $ G $ such that $ \ord (\vC_{i})= p^{n} $ for some $ n $, for all $ i\in \{1, 2, \dots, r \} $, and proved that $ G $ is either a $ p $-group or $ G $ has a homomorphic image which is a Frobenius group with a complement of $ p $-power order. Robati \cite{Rob19} recently proved that if $ \Van(G) $ contains three conjugacy classes of $ G $, then the group $ G $ is solvable.

In this article, we investigate finite groups $ G $ with the property below:

\[\mbox{\emph{$ \gcd(\ord(v\mathcal{C}_{i}),\ord(v\mathcal{C}_{j}))=1 $ for $ i\neq j $,  $ i,j\in \{1,2, \dots, r\} $.} \label{e:star} \tag{$\star$}}\] 

We also investigate finite groups $ G $ with a more general property:

\[\mbox{\emph{$ \gcd(\ord(v\mathcal{C}_{i}),\ord(v\mathcal{C}_{j}))\leqslant 2 $ for $ i\neq j $,  $ i,j\in \{1,2, \dots, r\} $.} \label{ee:star} \tag{$\star\star$}}\] 

In particular, using the classification of finite simple groups we show that if $ G $ has property \eqref{ee:star}, then $ G $ is solvable :

\begin{thmA}
Let $ G $ be a finite group. If $ G $ satisfies property \eqref{ee:star}, then $ G $ is solvable.
\end{thmA}

\subsection*{Remark} If $ \gcd(\ord(v\mathcal{C}_{i}),\ord(v\mathcal{C}_{j}))\leqslant 3 $, then $ G $ is not necessarily solvable for $ \SSS_{5} $ satisfies this property. Let $ \mathrm{Vo}(G) $ be the set of orders of vanishing elements of $ G $. Then if for every $ a,b \in \mathrm{Vo}(G) $, $ \gcd(a,b)= 1 $, then $ G $ is also not necessarily solvable: $ \A_{5} $ is a counterexample. 

A theorem of Thompson \cite[Corollary 12.2]{Isa06} states that, given a prime number $ p $, if every character degree of a non-linear character of $ G $ is a multiple of $ p $, then the group $ G $ has a normal $ p $-complement.  In \cite[Corollary B]{DPSS09}, it was shown that if $ G $ is a finite group and if $ p\mid a $ for all $ a \in \mathrm{Vo}(G) $ for some fixed prime $ p $, then $ G $ has a normal nilpotent $ p $-complement. This does not necessarily hold when $ G $ satisfies property \eqref{ee:star}. An example is $ \SSS_{4} $, since $ \gcd(\ord(v\mathcal{C}_{i}),\ord(v\mathcal{C}_{j}))\leqslant 2 $ for all $ v\mathcal{C}_{i}, v\mathcal{C}_{j}\subseteq \Van(\SSS_{4}) $, that is, $ \SSS_{4} $ satisfies property \eqref{ee:star} but $ \SSS_{4} $ does not have a normal $ 2 $-complement or $ 3 $-complement. However, if $ G $ is supersolvable or $ \textbf{O}_{2}(G)=1 $, then $ G $ has a normal $ 2 $-complement with one exception: some Frobenius groups with a homomorphic image isomorphic to $ \SSS_{4} $, as the following result states: 

\begin{thmB}
Let $ G $ be a finite non-abelian group satisfying property \eqref{ee:star}.
\begin{itemize}
\item[(a)] If $ G $ is supersolvable, then $ G $ has a normal metabelian $ 2 $-complement
\item[(b)] If $ \textbf{O}_{2}(G)=1 $, then either
\begin{itemize}
\item[(i)] $ G $ has a normal $ 2 $-complement of Fitting height at most $ 3 $, or
\item[(ii)] $ G $ is a Frobenius group which has an abelian kernel and a Frobenius complement isomorphic to $ \mathrm{S}_{4} $.
\end{itemize}
\end{itemize}
\end{thmB}

In \cite[Proposition 2.7]{Chi99}, Chillag showed that if $ G $ is a non-abelian group, then $ G $ is a Frobenius group with an abelian odd order kernel and a complement of order $ 2 $ if and only if every irreducible character of $ G $ vanishes on at most one conjugacy class. In this article, we prove a new characterisation of these Frobenius groups:

\begin{corC}
Let $ G $ be a finite non-abelian group. Then $ G $ has property \eqref{e:star} if and only if $ G $ is a Frobenius group with an abelian kernel and complement of order two.
\end{corC}
\section{Preliminaries}
In this section we shall list some properties of vanishing elements needed to prove our results. 
\begin{lemma}\label{G/Nsatifies}
Let $ G $ be an finite group and let $ N $be a normal subgroup of $ G $. Then the following statements hold:
\begin{itemize}
\item[(a)] If $ G $ satisfies property \eqref{e:star}, then $ G/N $ satisfies property \eqref{e:star}.
\item[(b)] If $ G $ satisfies property \eqref{ee:star}, then $ G/N $ satisfies property \eqref{ee:star}.
\end{itemize}
\end{lemma}
\begin{proof}
The result follows by the standard observation that $ xN\in \Van(G/N) $ implies that $ xN\subseteq \Van(G) $. 
\end{proof}

\begin{lemma}\cite[Lemma 2]{Qia02}\label{QiaLemma2}
Let $ G $ be a finite solvable group. Suppose $ M, N $ are normal subgroups of $ G $.
\begin{itemize}
\item[(a)] If $ M\setminus N $ is a conjugacy class and $ \gcd(|M{:}N|, |N|)=1 $, then $ M $ is a Frobenius group with kernel $ N $ and prime order complement.
\item[(b)] If $ G\setminus N $ is a conjugacy class, then $ G $ is a Frobenius group with an abelian kernel and complement of order two.
\end{itemize} 
\end{lemma}

For a positive integer $ m $, set $ \pi(m):=\{p\mid p $ divides $ m $, where $ p $ is prime$ \} $.

\begin{corollary}\cite[Corollary 2.6]{DPSS10a}\label{DPSS10aCorollary2.6}
Let $G$ be a finite group and let $K$ be a nilpotent normal subgroup of $G$. If $ K \cap \Van(G) \neq \emptyset $, then there exists $ g\in K\cap \Van(G)$ whose order is divisible by every prime in $ \pi(|K|) $.
\end{corollary}


\begin{lemma} \cite[Theorem D]{INW99}\label{INW99TheoremD} Let G be a finite solvable group. If $x$ is a non-vanishing element of $G$, then $ x\mathbf{F}(G) $ is a $2$-element of $ G/\mathbf{F}(G) $. If $ G $ is not nilpotent, then $ x $ lies in the penultimate term of the Fitting series.
\end{lemma}

A non-linear irreducible character $ \chi $ of G is said to be of $ p $-defect zero if $ p $ does not divide $ |G|/\chi(1) $. By a result of Brauer (see \cite[Theorem 8.17]{Isa06}), if $ \chi $ is an
irreducible character of $ p $-defect zero of $ G $, then $ \chi(g) = 0 $ whenever $ p $ divides the order of $ g $ in $ G $. The existence of $ p $-defect zero characters is guaranteed in finite simple groups $ G $ for almost all primes $ p\geqslant 5 $ dividing $ |G| $ as the following result shows:

\begin{lemma}\cite[Corollary 2.2]{GO96}\label{pdefectzero5ormore} Let $ G $ be a non-abelian finite simple group and $ p $ be a prime. If $ G $ is a finite group of Lie type, or if $ p\geqslant 5 $, then there exists $ \chi \in \Irr(G) $ of $ p $-defect zero. 
\end{lemma}

\begin{lemma}\cite[Lemma 2.2]{Bro16}\label{Bro16Lemma2.2} Let $ G $ be a finite group, $ N $ a normal subgroup of $ G $ and $ p $ be a prime. If $ N $ has an irreducible character of $ p $-defect zero, then every element of $ N $ of order divisible by $ p $ is a vanishing element in $ G $.
\end{lemma}

\begin{lemma}\cite[Lemma 5]{BCLP07}\label{BCLP07Lemma5} Let $ G $ be a finite group, and $ N=S_{1}\times \cdots \times S_{k} $ a minimal normal subgroup of $ G $, where every $ S_{i} $ is isomorphic to a non-abelian simple group $ S $. If $ \theta \in \Irr(S) $ extends to $ \Aut(S) $, then $ \varphi= \theta \times \cdots \times \theta \in \Irr(N) $ extends to $ G $.
\end{lemma}

\begin{lemma}\cite[Theorem 1.1]{MT-V11}\label{MT-V11Theorem1.1}
Suppose that $ N $ is a minimal normal non-abelian subgroup of a finite group $ G $. Then there exists an irreducible character $ \theta $ of $ N $ such that $ \theta $ is extendible to $ G $ with $ \theta(1)\geqslant 5 $.
\end{lemma}

%


The number theory result below follows easily.
\begin{lemma}\label{outerlessthanconjugacyclasses}
Let $ p $ be and $ f $ be a positive integer. If $ q=p^{f}\geqslant 32 $, then $ f < (q-2)/2 $.
\end{lemma}
We end this section by stating a result on groups in which every irreducible character vanishes on at most two conjugacy classes.

\begin{theorem}\cite[Theorem 1]{BCG00}\label{BCG00Theorem1}
Let $ G $ be a non-abelian finite group in which every irreducible character vanishes on at most two conjugacy classes. Then one of the following holds:
\begin{itemize}
\item[(a)] $ G\cong \mathrm{A}_{5} $ or $ G\cong \PSL_{2}(7) $;
\item[(b)] $ G $ is solvable and one of the following holds:
\begin{itemize}
\item[(i)] $ G $ has a subgroup $ Z $ with $ |Z|\leqslant 2 $ such that $ G/Z $ is Frobenius group with a Frobenius complement of order $ 2 $ and an abelian Frobenius kernel of odd order.
\item[(ii)] $ G/Z=FA $ is a semidirect product, where $ |A|\leqslant 2 $, $ |Z|\leqslant 2 $ and $ F $ is a Frobenius group with a Frobenius complement of order $ 3 $ and a nilpotent Frobenius kernel of class at most $ 2 $.
\end{itemize}
\end{itemize}
\end{theorem}


%

\section{Theorem A}

\begin{proof}[\textbf{Proof of Theorem A}] We prove the result by induction on $ |G| $. Let $ N $ be a non-trivial normal subgroup of $ G $. Then $ G/N $ satisfies property \eqref{ee:star} by Lemma \ref{G/Nsatifies}(b) and hence $ G/N $ is a solvable group. If $ N_{1} $ and $ N_{2} $ are two minimal normal subgroups of $ G $, then $ G/N_{1} $ and $ G/N_{2} $ are solvable. Hence $ G $ is solvable. We may assume that $ G $ has a unique non-abelian minimal normal subgroup $ N $. If $ N=S_{1}\times S_{2}\times \cdots \times S_{k} $,  where $ S_{i}\cong S $, $ S $ is a simple group and $ i=1,2, \dots ,k $, then by Lemma \ref{MT-V11Theorem1.1}, there exists $ \theta \in \Irr(N) $ which is extendible to $ G $. Note that $ \theta = \phi_{1} \times \phi_{2} \times \cdots \times \phi_{k}  $ with $ \phi_{i} \in \Irr(S_{i}) $ for each $ i\in \{1,2, \dots ,k\} $. Suppose that $ k\geqslant 2 $. Since $ \phi_{1} $ is non-linear, we may assume that $ \phi_{1} $ vanishes on a $ p $-element $ x_{1}\in S_{1} $ for some prime $ p $ by \cite[Theorem B]{MNO00}. Suppose that $ p $ is odd.  Note that $ 2\mid |N| $ and let $ y_{2} \in S_{2}$ be a $ 2 $-element. Then $ \theta(x_{1}y_{2})=\phi_{1}(x_{1})\phi_{2}(y_{2})\cdots\phi_{k}(1)=0 $ and $ p\mid \gcd(\ord(x_{1}), \ord(x_{1}y_{2})) $. Hence $ G $ does not satisfy \eqref{ee:star}. Suppose that $ p $ is even. Then there is a prime $ q\geqslant 5 $ such that $ q\mid |N| $ since by \cite[Theorem 3.10]{Isa06}, $ \pi(|G|)\geqslant 3 $. Let $ y_{2}\in S_{2} $ be a $ q $-element. Note that $ x_{1}y_{2} $ and $ y_{2} $ are vanishing elements of $ G $ by Lemma \ref{pdefectzero5ormore} and Lemma \ref{Bro16Lemma2.2}. Since $ q\mid \gcd(\ord(y_{2}), \ord(x_{1}y_{2})) $, the result follows.

We may assume that $ N$ is a simple group. Since $ N $ is the unique minimal normal subgroup of $ G $, $ \mathbf{C}_{G}(N)=1 $ and so $ G $ is almost simple. Let $ N$ be a sporadic simple group or $ ^{2}\mathrm{F}_{4}(2)' $. Table \ref{SporadicTable} below contains an irreducible character $ \theta $ of $ N $ of $ p $-defect zero for some odd prime $ p $ and two elements of distinct orders divisible by $ p $. The result that $ G $ does not satisfy property \eqref{ee:star} follows from Lemma \ref{Bro16Lemma2.2}. We shall use the character tables and notation in the \Atlas{} \cite{CCNPW85}.

\begin{center}\label{SporadicTable}
\begin{tabular}{|c|c|c|c|}
\cline{1-4}
$ N $ & $ \theta(1) $ & $ \ord(v\mathcal{C}_{1}) $  & $ \ord(v\mathcal{C}_{2}) $ \\
\cline{1-4}
$ M_{11} $  &   $ \chi_{9}(1)=45 $   &   $ 3A $    &  $ 6A $ \\
\cline{1-4}
$ M_{12} $  &   $ \chi_{8}(1)=55 $   &   $ 5A $    &  $ 10A $ \\
\cline{1-4}
$ M_{22} $  &   $ \chi_{3}(1)=45 $   &   $ 3A $    &  $ 6A $\\
\cline{1-4}
$ M_{23} $  &   $ \chi_{3}(1)=45 $   &   $ 3A $   &  $ 6A $ \\
\cline{1-4}
$ M_{24} $  &   $ \chi_{3}(1)=45 $   &   $ 5A $   &  $10A$ \\
\cline{1-4}
$ J_{2} $  &   $ \chi_{18}(1)=225 $   &   $ 5A $   &  $ 10A $ \\
\cline{1-4}
$ Suz $  &   $ \chi_{10}(1)=10725 $   &   $ 5A $   &  $10A$ \\
\cline{1-4}
$ HS $  &   $ \chi_{19}(1)=1750 $   &   $ 5A $   &  $ 10A $ \\
\cline{1-4}
$ M^{c}L $  &   $ \chi_{9}(1)=1750 $   &   $ 5A $   &  $ 10A $\\
\cline{1-4}
$ Co_{3} $  &   $ \chi_{6}(1)=896 $   &   $ 7A $   &  $ 14A $ \\
\cline{1-4}
$ Co_{2} $  &   $ \chi_{5}(1)=1771 $   &   $ 7A $   &  $ 14A $ \\
\cline{1-4}
$ Co_{1} $  &   $ \chi_{4}(1)=1771 $   &   $ 11A $   &  $ 22A $ \\
\cline{1-4}
$ He $  &   $ \chi_{9}(1)=1275 $   &   $ 5A $   &  $ 10A $ \\
\cline{1-4}
 $ Fi_{22} $ & $ \chi_{4}(1)=1001 $ & $ 7A $ & $ 14A $ \\
\cline{1-4}
$ Fi_{23} $ & $ \chi_{8}(1)=106743 $ & $ 7A $ & $ 14A $ \\
\cline{1-4}
$ Fi_{24}' $ & $ \chi_{4}(1)=249458 $ & $ 11A $ & $ 22A $\\
\cline{1-4}
$ HN $ & $ \chi_{2}(1)=133 $ & $ 7A $ & $ 14A $ \\
\cline{1-4}
$ Th $ & $ \chi_{6}(1)=30628 $ & $ 13A $ & $ 39A $\\
\cline{1-4}
$ B $ & $ \chi_{11}(1)=3214743741 $ & $ 11A $ & $ 22A $\\
\cline{1-4}
$ M $ & $ \chi_{9}(1)=36173193327999 $ & $ 17A $ & $ 34A $ \\
\cline{1-4}
$ J_{1} $ & $ \chi_{9}(1)=120 $ & $ 3A $ & $ 6A $\\
\cline{1-4}
$ O'N $ & $ \chi_{8}(1)=32395 $ & $ 5A $ & $ 10A $\\
\cline{1-4}
$ J_{3} $ & $ \chi_{2}(1)=85 $ & $ 5A $ & $ 10A $ \\
\cline{1-4}
$ Ly $ & $ \chi_{5}(1)=48174 $ & $ 7A $ & $ 14A $\\
\cline{1-4}
$ Ru $ & $ \chi_{2}(1)=378 $ & $ 7A $ & $ 14A $ \\
\cline{1-4}
$ J_{4} $ & $ \chi_{9}(1)=1187145 $ & $ 5A $ & $ 10A $ \\
\cline{1-4}
$ ^{2}\mathrm{F}_{4}(2)' $ & $ \chi_{8}(1)=325 $ & $ 5A $ & $ 10A $ \\
\hline
\end{tabular}
\end{center}

Suppose that $ N $ is an alternating group $ \A_{n} $, $ n\geqslant 5 $. For $ N\cong \mathrm{A}_{5} $ and $ N\cong \mathrm{A}_{7} $, our result follows by consulting the \Atlas{} \cite{CCNPW85}. Assume that $ N\cong \mathrm{A}_{6}\cong \PSL_{2}(9) $. Then for an almost simple $ G $ such that $ |G/N|\leqslant 2 $, our result follows by consulting the explicit character tables in the \Atlas{} \cite{CCNPW85}.
For the case when $ |G/N|=4 $, we obtain the character table in \GAP{} \cite{GAP16} and our result follows.

Let $ N\cong \A_{n} $, $ n\geqslant 8 $. By Lemma \ref{pdefectzero5ormore}, $ N $ has a $ 5 $-defect zero character. Note that $ N $ has two elements $ (12345) $ and $ (12345)(678) $ of orders $ 5 $ and $ 15 $, respectively. These two elements are vanishing elements of $ G $ by Lemma \ref{Bro16Lemma2.2} and hence $ G $ does not satisfy property \eqref{ee:star}.

Suppose that $ N $ has an element of order $ 2r $, $ r $ an odd prime. Then $ N $ has an irreducible character $ \theta $ of $ r $-defect zero by Lemma \ref{pdefectzero5ormore}. Since $ N $ has elements of order $ r $ and $ 2r $, the result follows since these two elements are vanishing elements by Lemma \ref{Bro16Lemma2.2}. We may assume that $ N $ has no element of order $ 2r $, $ r $ an odd prime. Then the centralizer of each involution contained in $ N $ is a $ 2 $-group. It follows from \cite[III, Theorem 5]{Suz61} that $ N $ is isomorphic to one of the following: $ \PSL_{2}(p) $, where $ p $ is a Fermat or Mersenne prime; $ \PSL_{2}(9) $; $ N\cong \PSL_{3}(4) $; $ N\cong ^{2}\!\!\mathrm{B}_{2}(2^{2n + 1}) $, $ n\geqslant 1 $. 

Thus far, we have dealt with the cases when $ N\cong \PSL_{2}(5)\cong \mathrm{A}_{5} $ and $ N\cong \PSL_{2}(9)\cong \mathrm{A}_{6} $. For $ N\cong \PSL_{2}(7) $ the result follows by checking the character tables in the \Atlas{} \cite{CCNPW85}. Suppose that $ N\cong \PSL_{2}(p) $, where $ p \geqslant 17 $, is a Fermat or Mersenne prime. Then the centralizer of an involution in $ N $ is a dihedral group of order $ 2^{n} $, $ n\geqslant 4 $. Hence $ N $ contains elements of order $ 4 $ and $ 8 $. Using Lemmas \ref{pdefectzero5ormore} and \ref{Bro16Lemma2.2}, we conclude that $ G $ does not satisfy property \eqref{ee:star}.

Suppose $ N\cong \PSL_{3}(4) $.  Then for an almost simple $ G $ such that $ |G/N|\leqslant 6 $, our result follows by checking the explicit character tables in the \Atlas{} \cite{CCNPW85}. For the case when $ |G/N|=12 $, we obtain the character table from \GAP{} \cite{GAP16} and by checking the pertinent information, our result follows.

We may assume that $ N\cong ^{2}\!\!\mathrm{B}_{2}(2^{2n+1}) $ with $ n\geqslant 1 $. The result basically follows from \cite[Proposition 3.13]{Mad19zp} but we shall prove it here for completeness. Now $ N $ has two conjugacy classes of elements of order $ 4 $ by \cite[Proposition 18]{Suz62}. Since the outer automorphism group is cyclic of odd order $ 2n + 1 $, the outer automorphisms cannot fuse these two conjugacy classes to one conjugacy class in $ G $. Hence $ G $ has two conjugacy classes of order $ 4 $ and so $ G $ does not satisfy property \eqref{ee:star}. This concludes our argument.
\end{proof}

\section{Normal $ 2 $-complements}
Given a finite set of positive integers $Y,$ the prime graph $ \Pi (Y) $ is defined as the undirected graph whose vertices are the primes $ p $ such that there exists an element of $ Y $ divisible by $ p $, and two distinct vertices $ p, q $ are adjacent if and only if there exists an element of $ Y $ divisible by $ pq $. The vanishing prime graph of $ G $, denoted by $ \Gamma(G) $, is the prime graph $ \Pi(\mathrm{Vo}(G)) $. We shall state a result on solvable groups with disconnected vanishing prime graphs. We first recall two definitions: 

A group $ G $ is said to be a \emph{$ 2 $-Frobenius group} if there exists two normal subgroups $ F $ and $ L $ of $ G $ such that $ G/F $ is a Frobenius group with kernel $ L/F $ and $ L $ is a Frobenius group with kernel $ F $. 

A group $ G $ is said to be a \emph{nearly $ 2 $-Frobenius group} if there exist two normal subgroups $ F $ and $ L $ of $ G $ with the following properties: $ F=F_{1}\times F_{2} $ is nilpotent, where $ F_{1} $ and $ F_{2} $ are normal subgroups of $ G $. Furthermore, $ G/F $ is a Frobenius group with kernel $ L/F $, $ G/F_{1} $ is a Frobenius group with kernel $ L/F_{1} $, and $ G/F_{2} $ is a $ 2 $-Frobenius group.

\begin{theorem} \cite[Theorem A]{DPSS10b} \label{DPSS10bTheoremA}
Let $ G $ be a finite solvable group. Then $ \Gamma(G) $ has at most two connected components. Moreover, if $ \Gamma(G) $ is disconnected, then $ G $ is either a Frobenius group or a nearly $ 2 $-Frobenius group.
\end{theorem}

The following is a classification of Frobenius complements.

\begin{theorem}\cite[Theorem 1.4]{Bro01}\label{Bro01Theorem1.4}
Let $ G $ be a Frobenius group with Frobenius complement $ M $. Then $ M $ has a normal subgroup $ N $ such that all Sylow subgroups of $ N $ are cyclic and one of the following holds:
\begin{itemize}
\item[(a)] $ M/N \cong 1 $;
\item[(b)] $ M/N \cong \mathrm{V}_{4} $, the Sylow $ 2 $-subgroup of the alternating group $ \mathrm{A}_{4} $;
\item[(c)] $ M/N \cong \mathrm{A}_{4} $;
\item[(d)] $ M/N \cong \mathrm{S}_{4} $;
\item[(e)] $ M/N \cong \mathrm{A}_{5} $;
\item[(f)] $ M/N \cong \mathrm{S}_{5} $.
\end{itemize}
\end{theorem}

\begin{proof}[\textbf{Proof Theorem B}]
We first assume that $ G $ satisfies property \eqref{e:star}.
Since $ G $ is solvable, $ G $ contains at most two vanishing conjugacy classes by Theorem \ref{DPSS10bTheoremA}. If $ G $ has one vanishing class, then every irreducible character of $ G $ vanishes on at most one conjugacy class and by \cite[Proposition 2.7]{Chi99}, $ G $ is a Frobenius group with a Frobenius complement of order $ 2 $ and an odd order kernel. Hence $ G $ has an abelian normal $ 2 $-complement. Suppose that $ G $ has exactly two conjugacy classes. Note that $ \ord(\vC_{1}) \not= \ord(\vC_{2}) $. Then every irreducible character of $ G $ vanishes on at most two conjugacy classes. Using Theorem \ref{BCG00Theorem1}, we have two cases. Suppose that Theorem \ref{BCG00Theorem1}(b)(ii) holds. Then $ G $ has at least two vanishing conjugacy classes of order $ 3 $, a contradiction. Assume that Theorem \ref{BCG00Theorem1}(b)(i) holds. Then $ G $ has one vanishing conjugacy class with elements of order $ 2 $ contained in the Frobenius complement of $ G/Z $. By Lemma \ref{QiaLemma2}(b), $ G $ is a Frobenius group with an abelian kernel and complement of order two, that is, $ Z=1 $ and the result follows. 

Assume that $ G $ satisfies property \eqref{ee:star} and suppose $ \gcd(\ord(v\mathcal{C}_{i}),\ord(v\mathcal{C}_{j}))=2 $ for some $ i\neq j $. Let $ G $ be nilpotent. Then $ G=P_{2}\times H $ with $ P_{2} $, the Sylow $ 2 $-subgroup of $ G $ and $ H $, a nilpotent group of odd order. If $ H $ is non-abelian, then $ H $ has a vanishing $ h $ of $ G $ by \cite[Theorem B]{INW99}. This means that $ \chi(h)=\theta_{1}(1)\times \theta_{2}(h)=\theta_{2}(h)=0 $ for some $ \chi=\theta_{1}\times \theta_{2}\in \Irr(P_{2}\times H) $ with $ \theta_{1}\in \Irr(P_{2}) $ and $ \theta_{2}\in \Irr(H) $. It follows that for any $ x\in P_{2} $, $ \chi(xh)=\theta_{1}(x)\times \theta_{2}(h)=0 $ and so $ xh $ and $ h $ are vanishing elements of $ G $, a contradiction. Thus $ H $ is abelian which implies that $ G $ has a normal abelian $ 2 $-complement.

Let $ 2=p_{1} < p_{2} < \dots < p_{n} $ and let $ P_{i} $ be a Sylow $ p_{i} $-subgroup of $ G $ for $ i\in \{1,2, \dots, n \} $. Suppose  that $ G $ is a non-nilpotent supersolvable group and consider $ \mathbf{F}(G) $. If $ \pi (|\mathbf{F}(G)|)=1 $, then by \cite[Theorems 6.2.5 and 6.2.2]{Bec71}, $ \mathbf{F}(G) $ is the $ p_{n} $-subgroup $ P_{n} $ of $ G $. Let $ M $ be a normal subgroup of $ G $ such that $ M/\mathbf{Z}(\mathbf{F}(G)) $ is a chief factor of $ G $. Then $ M\setminus \mathbf{Z}(\mathbf{F}(G)) $ is a conjugacy class of $ G $ since $ M\setminus \mathbf{Z}(\mathbf{F}(G))\subseteq \Van(G) $ and $ \gcd(\ord(v\mathcal{C}_{i}),\ord(v\mathcal{C}_{j}))\leqslant 2 $. Hence $ M/\mathbf{Z}(\mathbf{F}(G)) $ is cyclic and so $ M $ is abelian. Thus $ M=\mathbf{F}(G)=P_{n} $.

Suppose that $ p_{n-1} $ is an odd prime. Then $ P_{n-1}\mathbf{F}(G)/\mathbf{F}(G) $ is a normal subgroup of $ G/\mathbf{F}(G) $. Thus $ P_{n-1}\mathbf{F}(G)\setminus \mathbf{F}(G) $ is a conjugacy class since $ \gcd(\ord(v\mathcal{C}_{i}),\ord(v\mathcal{C}_{j}))\leqslant 2 $. By Lemma \ref{QiaLemma2}(b), $ \mathbf{F}(G)P_{n-1}P_{n-2} $ is a Frobenius group of kernel $ \mathbf{F}(G)P_{n-1} $ and complement of order $ p_{n-2} $. This means that the kernel $ \mathbf{F}(G)P_{n-1} $ is nilpotent, that is, $ \mathbf{F}(G)P_{n-1}=\mathbf{F}(G)\times P_{n-1} $. Since $ P_{n-1}P_{n-2}\cdots P_{1} $ is supersolvable and since by \cite[Theorems 6.2.5 and 6.2.2]{Bec71}, $ P_{n-1} $ is normal in $ P_{n-1}P_{n-2}\cdots P_{1} $, we obtain that $ P_{n-1} $ is a normal subgroup of $ G $, a contradiction since $ \mathbf{F}(G)=P_{n} $. Then $ P_{n-2}=P_{1} $ is a Sylow $ 2 $-subgroup of $ G $ and so $ \pi(|G|)\leqslant 3 $. Hence $ \mathbf{F}(G)P_{n-1} $ is a metabelian normal $ 2 $-complement of $ G $.

Suppose that $ \pi(|\mathbf{F}(G)|)\geqslant 2 $. Let $ \mathbf{F}(G)=Q_{1}\times Q_{2}\times \cdots \times Q_{n} $, where $ 2=q_{1} < q_{2} < \dots < q_{n} $ and let $ Q_{i} $ be the Sylow $ q_{i} $-subgroup of $ \mathbf{F}(G) $ for $ i\in \{1,2, \dots, m \} $. Note that $ Q_{i}\setminus \mathbf{Z}(Q_{i})\subseteq \Van(G) $ by \cite[Theorem B]{INW99} since $ G $ is supersolvable and so all the non-vanishing elements of $ G $ are contained in $ \mathbf{Z(\mathbf{F}(G))}=\mathbf{Z}(Q_{1})\times \mathbf{Z}(Q_{2})\times \cdots \times \mathbf{Z}(Q_{n}) $. If $ q_{i} $ is an odd prime and there exists a $ q_{i} $-element $ m $ which is a vanishing element, then by Corollary \ref{DPSS10aCorollary2.6}, there exists a vanishing element $ mn $ whose order is divisible by every prime in $ \pi(|\mathbf{F}(G)|) $, a contradiction since $ \gcd(\ord(m),\ord(mn)) > 2 $. Hence $ Q_{i} $ is abelian for all odd $ q_{i} $'s. Consider $ Q_{i} $, the Sylow $ 2 $-subgroup of $ \mathbf{F}(G) $. If $ Q_{1}\setminus \mathbf{Z}(Q_{1}) $ has an element of order greater than $ 2 $, then by Corollary \ref{DPSS10aCorollary2.6} and the above argument, we obtain a contradiction. Then $ Q_{1}\setminus \mathbf{Z}(Q_{1}) $ consists only of involutions. By \cite[Theorem C]{BDS09}, $ Q_{1} $ is a direct product of an elementary abelian $2$-group and a Frobenius group with Frobenius complement of order $2$, a contradiction. Hence $ Q_{1} $ is abelian and therefore $ \mathbf{F}(G) $ is abelian. Thus $ G $ is metabelian. 

Finally suppose $ 2=p_{1} < p_{2} < \dots < p_{n} $ and let $ P_{i} $ be a Sylow $ p_{i} $-subgroup of $ G $ for $ i\in \{1,2, \dots, n \} $. By \cite[Theorems 6.2.5 and 6.2.2]{Bec71}, $ H=P_{n}P_{n-1}\dots P_{2} $ is a normal subgroup of $ G $ and therefore $ H $ is a metabelian $ 2 $-complement as required.

We may now assume that $ G $ is not supersolvable and $ \textbf{O}_{2}(G)=1 $. Suppose that $ \Gamma(G) $ consists of a single vertex of an odd prime. Since $ G $ satisfies property \eqref{ee:star}, we have that $ G $ has one vanishing conjugacy class. This means that every irreducible character of $ G $ vanishes on at most one conjugacy class and by \cite[Proposition 2.7]{Chi99}, $ G $ is a Frobenius group with a Frobenius complement of order $ 2 $, a contradiction. So if $ \Gamma(G) $ is connected, then property \eqref{ee:star} implies that every element of $ \mathrm{Vo}(G) $ is divisible by $ 2 $. By \cite[Corollary B]{DPSS09}, $ G $ has a normal nilpotent $ 2 $-complement as required. We may assume that $ \Gamma(G) $ is disconnected. Then by Theorem \ref{DPSS10bTheoremA}, $ G $ is either a Frobenius group or a nearly $ 2 $-Frobenius group. 

Suppose that $ G $ is a Frobenius group. Assume further that the Frobenius complement of $ G $ has odd order. Let $ H $ be a maximal subgroup of $ G $ that contains $ G' $. Then $ |G/H|=p $ for some odd prime $ p $. Note that $ \mathbf{F}(G) \leq H $. Hence $ G\setminus H\subseteq \Van(G) $. Since $ p $ divides the order of every element in $ G\setminus H $. By Lemma \ref{QiaLemma2}(a), $ G\setminus H $ has at least two conjugacy classes, $ v\mathcal{C}_{1} $ and $ v\mathcal{C}_{2} $ say, and $ p\mid \gcd(\ord(v\mathcal{C}_{1}),\ord(v\mathcal{C}_{2})) $, contradicting our hypothesis. We may assume that the Frobenius complement has even order. Denote it by $ M $. Then the Frobenius kernel $ K $ is abelian. Using Theorem \ref{Bro01Theorem1.4}, we have that $ M $ has a unique normal subgroup $ N $ such that all the Sylow subgroups of $ N $ are cyclic and $ M/N\in \{1, \mathrm{V}_{4}, \mathrm{A}_{4}, \mathrm{S}_{4}, \mathrm{A}_{5}, \mathrm{S}_{5} \} $. Since $ G $ is solvable we need not consider $ \mathrm{A}_{5}, \mathrm{S}_{5} $. Note that $ N $ is metacyclic and supersolvable by \cite[p. 290]{Rob95}. Suppose that $ M/N\in \{1, \mathrm{V}_{4}\} $. Then since $ N $ is supersolvable, the Hall $ 2' $-subgroup $ R $ of $ N $ is normal in $ N=M $. Now $ KR $ is a normal $ 2 $-complement of $ G $. Also note that $ KR $ is of derived length at most $ 3 $ and thus of Fitting height at most $ 3 $, as required. Suppose $ M/N\cong \mathrm{A}_{4} $. Then there exists a normal subgroup $ T $ of $ G $ such that $ |G/T|=3 $. The result follows using the argument above in the case when the Frobenius complement is of odd order. We now suppose that $ M/N\cong \mathrm{S}_{4} $. Note that $ \Gamma(G) $ has two connected components and vertex $ 3 $ is isolated. A Sylow $ 2 $-subgroup $ T $ of $ G $ is a generalized quaternion. This means that $ |T|=8 $ and $ |N| $ is of odd order, otherwise $ G $ does not satisfy property \eqref{ee:star}. If there is a prime $ r\neq 3 $ such that $ r\mid |N| $, then using \cite[Proposition 3.2]{DPSS10b}, there exists an vanishing element $ g $ such that $ \ord(g) $ is either divisible by $ 3s $ or $ rs $ for some prime $ s $ such that $ s\mid |M| $, a contradiction. But that means the Frobenius complement has a cyclic Sylow $ 3 $-subgroup of order $ k $ greater than $ 3 $. Then $ G $ has vanishing elements of orders $ 3 $ and $ k $, a contradiction. Hence $ N=1 $ and therefore $ G/K\cong \SSS_{4} $. The result follows.

Suppose that $ G $ is a nearly $ 2 $-Frobenius group. Then there exist two normal subgroups $ F $ and $ L $ of $ G $ with the following properties: $ F=F_{1}\times F_{2} $ is nilpotent, where $ F_{1} $ and $ F_{2} $ are normal subgroups of $ G $. Furthermore, $ G/F $ is a Frobenius group with kernel $ L/F $, $ G/F_{1} $ is a Frobenius group with kernel $ L/F_{1} $, and $ G/F_{2} $ is a $ 2 $-Frobenius group. Since $ G/F_{2} $ is a $ 2 $-Frobenius group and $ G/F $ is a Frobenius group with kernel $ L/F $, it follows that $ L/F_{2} $ is a Frobenius group with kernel $ F/F_{2} $. By \cite[Remark 1.2]{DPSS10b}, $ G/L $ is cyclic and $ L/F $ is cyclic with $ |L/F| $ odd. If $ |G/L| $ is odd, then using the argument in the first part of the previous paragraph, we obtain a contradiction. Hence we may assume that $ |G/L| $ is even. Since $ G/F_{1} $ is a Frobenius group with kernel $ L/F_{1} $, we conclude that $ L/F_{1} $ is nilpotent and $ \gcd(|G/L|, |L/F_{1}|)=1 $. Note that $ | F| $ is odd since $ \textbf{O}_{2}(G)=1 $. We consider $ G/F $, a Frobenius group with a cyclic kernel $ L/F $ and a cyclic Frobenius complement $ G/L $. It follows that the Hall $ 2' $-subgroup $ J/L $ of $ G/L $ is cyclic. Hence $ J/L $ is cyclic, $ L/F_{1} $ is nilpotent and $ F_{1} $ is nilpotent, that is $ J $ is a normal $ 2 $-complement of $ G $ with Fitting height at most $ 3 $. This concludes our proof.
\end{proof}

\begin{proof}[\textbf{Proof of Corollary C}]
If $ G $ satisfies property \eqref{e:star}, then $ G $ is a Frobenius group with an abelian kernel and complement of order two by the argument in the first paragraph of the proof of Theorem B. The converse holds because if $ G $ is a Frobenius group with an abelian kernel and complement of order two, then $ \Van(G) $ contains only one conjugacy class of elements of order two.
\end{proof}

\section*{Acknowledgements}
The authors would like to thank the reviewer for the careful reading of this article. Their comments and suggestions improved the presentation of the work.

\section*{Funding}
Sesuai Y. Madanha acknowledges the postdoctoral scholarship from University of KwaZulu-Natal. Bernardo G. Rodrigues acknowledges support of NRF through Grant Numbers 95725 and 106071.

\end{document}